\documentclass[reqno, english,11pt]{amsart}
\usepackage[]{amsaddr}

\usepackage[utf8]{inputenc} 

\usepackage{mathtools}

\usepackage{graphicx} 

\usepackage{color}
\usepackage[colorinlistoftodos,textsize=tiny]{todonotes}

\usepackage{array} 

\usepackage{a4wide}
\usepackage{enumerate}
\usepackage{amsfonts}
\usepackage{amssymb,amsmath,amsthm}

\input{widebar.tex}

\usepackage{hyperref}
\usepackage{enumitem}
\usepackage{bm} 
\usepackage[capitalize]{cleveref}

\DeclareMathOperator{\res}{res}
\newcommand{\PSL}[2]{\operatorname{PSL}(#1, \mathbb{F}_{#2})}

\newcommand{\abs}[1]{\left | #1 \right |}
\newcommand{\set}[1]{\left \{ #1 \right \}}
\newcommand{\NN}{\mathbb{N}_0}
\newcommand{\NNp}{\mathbb{N}}
\newcommand{\ZZ}{\mathbb{Z}}
\newcommand{\RR}{\mathbb{R}}
\renewcommand{\l}{\ell}
\newcommand{\OO}[2][]{O_{#1}\!\left( #2 \right)}
\renewcommand{\SS}{\mathfrak{S}}
\newcommand{\Ax}{A_{\leq x}}
\newcommand{\lfrac}[2]{#1 / #2}

\renewcommand{\epsilon}{\varepsilon}
\renewcommand{\phi}{\varphi}

\usepackage{xifthen}
\newcommand{\sums}[1]{\mathop{\sum\!\mathrlap{{\vphantom{\sum}}^*}}_{#1}\ifthenelse{\isempty{#1}}{\;\,}{}}

\newcommand{\preS}[1]{\prescript{#1\!}{}{S}}
\newcommand{\abg}{\alpha, \beta, \gamma}
\newcommand{\allA}{\widebar{A}}
\newcommand{\allB}{\widebar{B}}
\newcommand{\SSA}{\SS_{\!A}}
\newcommand{\SSnum}{5.71649719\text{\textellipsis}}

\makeatletter
\newcommand\footnoteref[1]{\protected@xdef\@thefnmark{\ref{#1}}\@footnotemark}
\makeatother

\newtheorem{thm}{Theorem}
\newtheorem{conj}{Conjecture}
\newtheorem{lem}[thm]{Lemma}

\theoremstyle{definition}
\newtheorem{defi}[thm]{Definition}
\newtheorem{rem}[thm]{Remark}

\newtheorem*{notation}{Notation}

\numberwithin{thm}{section}
\numberwithin{equation}{section}
\numberwithin{table}{section}

\begin{document}
\title{Estimating the density of a set of primes with applications to group theory}
\date{October 19, 2018} 
\author{Carlos Esparza}
\address{Technical University of Munich}
\email{carlos.esparza@tum.de}
\author{Lukas Gehring}
\address{University of Bonn}
\email{lukas.gehring@uni-bonn.de}

\subjclass[2010]{
    11N05 
}

\begin{abstract}
    We estimate the asymptotic density of the set $\allA$ of primes $p$ satisfying the constraint that $p+1$ and $p-1$ have only one prime divisor larger than 3. We also estimate the density of a maximal subset $\allB \subset \allA$ such that for $p_1, p_2 \in \allB$ no common prime divisor of $p_1(p_1 + 1)(p_1 - 1)$ and $p_2 (p_2 + 1)(p_2 - 1)$ is larger than $3$. Assuming a generalized Hardy--Littlewood conjecture, we prove that for both $\allA$ and $\allB$ the number of elements lesser than $x$ is asymptotically equal to a constant times $ x / (\log x)^3$.
\end{abstract}
    
\maketitle
    
\begin{notation}
    We consider $0 \not\in \NNp$.
    In the bounds of a sum or a product, $p$ always denotes a prime.
    $(a,b)$ without any function name is the greatest common divisor of $a$ and $b$.
    For a set $S \subset \NN$ and $x \in \RR$ we define $S_{\leq x} = \{n \in S : n \leq x\}$.
    $\log$ denotes the natural logarithm (unless subscripted with another base).
    
    We write $f(x) \ll g(x)$ with the same meaning as $f(x) = \OO{g(x)}$
    and subscript the $O$ or the $\ll$ with the quantities on which the implicit constants depend (no subscript means the implicit constant is absolute).
    Finally, $f(x) \sim g(x)$ means that $\lim_{x \to \infty} f(x) / g(x) = 1$ or equivalently $f(x) = g(x) + o(g(x))$.
\end{notation}

\section{Introduction}

The authors of \cite{graph} work on bounding the number of vertices $V(G)$ of the \emph{character degree graph}%
\footnote{\label{note:1} The definition can be found in \cite{graph} but it will not be relevant for this paper.}
(or simply \emph{degree graph}) of a finite group $G$ in terms of its \emph{clique number}\footnoteref{note:1} $\omega(G)$.

In \cite[Theorem A]{graph} they prove the bound
$
|V(G)| \leq 3\omega(G) - 4 \label{eq:graph}
$
for $\omega(G) \geq 5$ and conjecture that it can be attained for every value of $\omega(G)$.

In order to provide such a $G$ for all $\omega(G) \geq 5$, they choose a suitable set of primes $\Pi = \{p_1, \dots, p_n\}$ and construct $G_\Pi = \PSL{2}{p_1}~\times~\dots~\times~\PSL{2}{p_n}$.

This group satisfies $\omega(G_\Pi) = n + 2$ and $|V(G_\Pi)| = 3n + 2$. Therefore it provides the desired equality for $\omega(G) \leq N + 2$ where $N$ is the cardinality of a maximal suitable set of primes.

We now specify which sets of primes are suitable. To this end, let
\begin{align}
\allA = \{ p \text{ prime} : ~& p \equiv 2 \pmod{3}, \quad \exists \alpha, \beta, \gamma, \mu, \nu, q, r \in \NNp: \notag\\
&\quad p+1 = 2^\alpha 3^\beta q^\mu, \quad p-1 = 2^\gamma r^\nu , \quad q, r \text{ primes with } (6,qr)=1 \} \label{eq:AAdef}  ~,\
\end{align}
so $\allA$ is the set of all primes $p$ such that a $p+1$ and $p-1$ have only one prime divisor larger than $3$.
Further 
set $\allB$ to be a maximal subset of $\allA$ such that $\forall p_1, p_2 \in \allB$ with $p_1 \neq p_2$ the only common divisors of $p_1(p_1 + 1)(p_1 - 1)$ and $p_2(p_2 + 1)(p_2 - 1)$ are products of powers of $2$ and powers of $3$. A suitable set of primes $\Pi$ is a finite subset of $\allB$.


Assuming the infinitude of $\allB$ the authors of \cite{graph} therefore come to the conclusion that their upper bound on $V(G_\Pi)$ is optimal.

In this paper we will achieve some estimates about the sets $\allA$ and $\allB$ using a generalized form of the Hardy-Littlewood conjecture (\cref{thm:hlc}), which will be presented in \cref{sect:ana}. It is a special case of the generalized Hardy-Littlewood conjecture used in \cite{greentao}, except for the fact that we impose some additional hypotheses concerning the uniformity of the error term.
In fact, we will not only deal with the topic of infinitude but will also evaluate the density of those sets in the natural numbers. We also give some unconditional upper bound on this density based on the work of \cite{halb}. Therefore this paper can be seen as the number theoretical contribution to the work in \cite{graph}.

\subsection*{Acknowledgement}
This paper was written in September 2018, when both authors were interns at the Max Planck Institute for Mathematics in Bonn. We want to thank the Institute for the opportunity to do this internship and Pieter Moree and Efthymios Sofos for providing the problem treated in this paper. We especially thank Efthymios for his support and guidance during the internship.

\section{Main results} \label{sect:main}

We first derive a formula for the density of primes in $\allA$ using the generalized Hardy--Littlewood conjecture
\begin{thm} \label{thm:bound}
    Assuming \cref{thm:hlc} we have
    \begin{equation}
    |\allA_{\leq x}| \sim \frac{\SSA}{2} \frac{x}{(\log x)^3}~,
    \end{equation}
\end{thm}
\noindent
where $\SSA = 9\prod_{p\geq 5}\left(1-\frac{3}{p}\right)\left(1-\frac{1}{p}\right)^{-3}\approx \SSnum$~.

\medskip
We also prove an unconditional upper bound, using the same value of $\SSA$:
\begin{thm}\label{thm:uncon}
    In the situation of \cref{thm:bound} we have, without assuming any conjectures
    \begin{equation*}
    \limsup_{x \to \infty} \frac{\abs{\allA_{\leq x}}}{x / (\log x)^3} \leq 24 \SSA ~.
    \end{equation*}
\end{thm}
Finally, we give an estimate on the density of the set $\allB$:
\begin{thm} \label{thm:A0}
    Assuming \cref{thm:bound} we have
    \begin{equation*}
    \abs{\allB_{\leq x}} \sim \abs{\allA_{\leq x}} \sim \frac{\SSA}{2} \frac{x}{(\log x)^3} ~.
    \end{equation*}
\end{thm}


\section{Elementary reductions} \label{sect:elem}

We will not deal with $\allA$ and $\allB$ directly here. Instead we define two slightly modified sets $A$ and $B$ and will prove in \cref{lem:same} that the theorems about $A$ and $B$ presented in \ref{sect:main} also hold for $\allA$ and $\allB$.

The important result of this section will be \cref{lem:irr}, which tells us that the Hardy--Littlewood conjecture is applicable to the sets $A^{\alpha,\beta,\gamma}$.

\begin{defi} \label{def:AB}
Let
\begin{align}
    A = \{ p \text{ prime} : ~& p \equiv 2 \pmod{3}, \quad \exists \alpha, \beta, \gamma,q,r \in \NNp: \notag\\
            &\quad p+1 = 2^\alpha 3^\beta q, \quad p-1 = 2^\gamma r,\quad q, r \text{ primes with } (6,qr)=1 \} \label{eq:Adef} ~.
\end{align}
and set $B$ to be a maximal subset of $A$ such that $\forall p_1, p_2 \in B$ with $p_1 \neq p_2$ the only common divisors of $p_1(p_1 - 1)(p_1 + 1)$ and $p_2(p_2 - 1)(p_2 + 1)$ are products of powers of $2$ and powers of $3$.

For $\alpha,\beta,\gamma\in\NNp$, $x\in\RR_{\geq 1}$ we set

\begin{align*}
A_{\leq x}^{\alpha,\beta,\gamma}=\{p \in A_{\leq x} : \quad&\exists q,r\in\NNp:\\
&\quad p+1 = 2^\alpha 3^\beta q, \quad p-1 = 2^\gamma r,\quad q, r \text{ primes with } (6,qr)=1\} ~.
\end{align*}
\end{defi}

\begin{lem}\label{lem:exp}
	If $A^{\alpha,\beta,\gamma}$ is not empty then 
    \begin{equation}
    \beta\geq 1, \quad \min(\alpha,\gamma) = 1, \quad \max(\alpha,\gamma)\geq 2. \label{eq:ok}
    \end{equation}
\end{lem}
\begin{proof}
	Let $p \in A^{\alpha,\beta,\gamma}$ y $p \equiv 2 \pmod{3}$ we get $3\mid p+1$, thus $\beta\geq 1$. Furthermore $p$ is odd (because $2\notin A$), so $p-1$ and $p+1$ are two consecutive even numbers. Therefore both are divisible by $2$, but exactly one of them is divisible by $4$. This yields the desired $\min(\alpha,\gamma) = 1$, $\max(\alpha,\gamma)\geq 2$.
\end{proof}
\cref{lem:exp} gives us two kinds of primes in $A$ which we will mostly consider separately. The calculations are similar in both cases $\alpha = 1$ and $\gamma = 1$, so we skip some of them for one of the cases.

The properties of the primes in $A$ give us some relations between $p$ and the corresponding big prime divisors $q$ of $p+1$ and $r$ of $p-1$, respectively.
\begin{lem}\label{lem:gamma}
	If $\gamma = 1$ and $p\in A$, then $p$ is of the form
	\begin{equation*}
    	p=2^\alpha 3^\beta n - 1
	\end{equation*}
	for some $n\in \NN$.The corresponding $q$ and $r$ in \cref{eq:Adef} are then given by
	\begin{align*}
    	q&=n,\\
    	r&=2^{\alpha - 1}3^\beta n - 1 ~.
	\end{align*}
\end{lem}
\begin{proof}
	By $p+1=2^\alpha 3^\beta q$ we immediately get the desired $n$ by setting $n=q$. The given form for $r$ results from
	\begin{equation*}
		r = \frac{p-1}{2}=\frac{2^\alpha 3^\beta n -2}{2}=2^{\alpha - 1}3^\beta n - 1 ~.
	\end{equation*}
\end{proof}
In the case $\gamma=1$, these linear functions will from now on be denoted by $p(n),q(n),r(n)$.
\begin{lem}\label{lem:alpha}
	Assume that $\beta,\gamma\in\ZZ$ satisfy $\beta\geq 1$, $\gamma\geq 2$. Then the system of congruences
	\begin{align*}
	s &\equiv -1 \pmod{3^\beta} ~,\\
	s &\equiv 1 \pmod{2^\gamma}
	\end{align*}
	has a unique solution $k\in \ZZ$ with $0<k<3^\beta2^\gamma$.
	
	If $\alpha = 1$ and $p\in A$, then $p$ is of the form
	\begin{equation*}
    	p=3^\beta 2^\gamma  n + k
	\end{equation*}
	for some $n\in \NN$.
	
	The corresponding $q$ and $r$ in \cref{eq:Adef} are then given by
	\begin{align*}
    	q&=2^{\gamma-1} n + \frac{k+1}{2\cdot 3^\beta} ~,\\
    	r&=3^\beta n   + \frac{k-1}{2^\gamma} ~.
	\end{align*}
	Furthermore, $\frac{k+1}{2\cdot 3^\beta}$ is an odd integer and $\frac{k-1}{2^\gamma}$ is an integer not divisible by $3$.
\end{lem}
\begin{proof}
	We know $p = 2\cdot3^\beta q -1 =2^\gamma r+1$. Thus we get
	\begin{align*}
    	p &\equiv -1 \pmod{2\cdot 3^\beta} ~,\\
    	p &\equiv 1 \pmod{2^\gamma} ~.
	\end{align*}
	But as $\gamma\geq 2$, $p \equiv 1 \pmod{2^\gamma}$ implies $p\equiv -1 \pmod{2}$. By the Chinese Remainder Theorem, $p\equiv -1 \pmod{2}$ and $p\equiv -1 \pmod{3^\beta}$ yield $p \equiv -1 \pmod{2\cdot 3^\beta}$, so we can simplify the first of the two given congruences for $p$ and replace it by $p\equiv -1 \pmod{3^\beta}$. Now the two moduli are coprime, so we can use the Chinese Remainder Theorem again. Therefore we get the desired unique $k$ and the given form of $p$.
	
	The given forms for $q$ and $r$ result from:
	\begin{align*}
		q &=\frac{p+1}{2\cdot 3^\beta}=\frac{3^\beta 2^\gamma  n + k+1}{2\cdot 3^\beta}=2^{\gamma-1} n + \frac{k+1}{2\cdot 3^\beta} ~,\\
		r &=\frac{p-1}{2^\gamma}=\frac{3^\beta 2^\gamma  n + k-1}{2^\gamma}=3^\beta n   + \frac{k-1}{2^\gamma} ~.
	\end{align*}
	As $k$ fulfills
	\begin{align*}
    	k &\equiv -1 \pmod{3^\beta} ~,\\
    	k &\equiv 1 \pmod{2^\gamma} ~,
	\end{align*}
	both of $\frac{k+1}{2\cdot 3^\beta}$ and $\frac{k-1}{2^\gamma}$ are integers. By \cref{lem:exp} we get $\beta\geq 1$ and $\gamma\geq 2$, thus $3\mid k+1$ and $4\mid k-1$. But this yields $3\nmid k-1$ and $4\nmid k+1$, so $\frac{k-1}{2^\gamma}$ can not be divisible by 3 and $\frac{k+1}{2\cdot 3^\beta}$ can not be divisible by 2.
\end{proof}
In the case $\alpha=1$, these linear functions will from now on be denoted by $p(n), q(n), r(n)$.
\begin{lem}\label{lem:irr}
	In the situation of both Lemma~\ref{lem:gamma} and \ref{lem:alpha} the linear functions $p(n),q(n),r(n)$ describing $p, q, r$ are irreducible and primitive polynomials (i.\,e.\ of the form $an+b$ with $(a,b)=1$). For all odd primes $\l$, there is no $n\in\NN$ such that more than one of $p(n),q(n),r(n)$ is divisible by $\l$.
\end{lem}
\begin{proof}
	Primitivity is obvious in the situation of \cref{lem:gamma} and follows immediately from the remarks on $\frac{k+1}{2\cdot 3^\beta}$ and $\frac{k-1}{2^\gamma}$ in the situation of \cref{lem:alpha}. Since they are linear primitivity implies irreducibility.
	
	By Lemmas~\ref{lem:gamma} and \ref{lem:alpha}, there exist $c_1,c_2\in \NN$ such that $c_1q(n)-1=p(n)=c_2r(n)+1$. This means that if $l\mid p(n)$ and $\l \mid q(n)$ or $\l \mid r(n)$ then also $\l \mid 1$ (because $p(n)$ and $q(n)$ or $p(n)$ and $r(n)$, respectively, have multiples which differ by 1) which is a contradiction. The same argument yields $(q(n),r(n))\geq 2$.
	
	Thus there can not be any odd prime number dividing more than one of $p(n),q(n),r(n)$. In fact, we can even get $(q(n),r(n))=1$ as in both Lemmas~\ref{lem:gamma} and \ref{lem:alpha} one polynomial out of $q$ and $r$ has only odd values.
\end{proof}


\section{Analytic estimates} \label{sect:ana}

In this section we will first present and give evidence for the generalized Hardy--Littlewood conjecture and then move on to prove Theorems~\ref{thm:bound}, \ref{thm:uncon} and \ref{thm:A0} for $A$ and $B$ instead of $\allA$ and $\allB$. Finally we will show in \cref{lem:same} that the estimates from these theorems also hold for $\allA$ and $\allB$.

We now want to estimate the density of $A$ in $\NNp$ using \cref{thm:hlc}. First we need to define a value $\SS$ appearing in the conjecture.
\begin{defi} \label{def:SS}
Let $R \in \NNp$ and $a_i, b_i \in \ZZ$ for $1 \leq i \leq R$ with $(a_i, b_i) = 1$. The \emph{Hardy--Littlewood singular series} $\SS$ corresponding to the $a_i, b_i$ is defined as
\begin{equation}
\SS = \prod_{p\text{ prime}} \left(1-\frac{\rho(p)}{p}\right)\left(1-\frac{1}{p}\right)^{-R} \label{eq:sing} ~,
\end{equation}
where we write $\rho(p)$ for the number of solutions of $\prod_{i=1}^{R}(a_in+b_i)=0$ for $n \in \ZZ / p\ZZ$.
\end{defi}
\begin{conj}[Generalized Hardy--Littlewood conjecture] \label{thm:hlc}
    Fix $R \in \NNp$ and $C > 0$. Let $a_i, b_i \in \ZZ$ with $(a_i, b_i) = 1$, $a_i>0$ for $1 \leq i \leq R$.
    If the product for $\SS$ in \cref{eq:sing} converges%
    \footnote{In particular this implies $a_i \neq a_j$ or $b_i \neq b_j$ for $i \neq j$.}%
    , we conjecture that
    \begin{equation*}
    \abs{\set{n \leq x : \forall i\colon a_i n + b_i \mathrm{~prime}}} = \left( 1 + o_{R, C}(1) \right) \SS \frac{x}{ (\log x)^R}
    + o_{R, C} \! \left( \frac{x}{(\log x)^{R}} \right)
    \end{equation*}
    for $x$ satisfying $(\log x)^C \geq |a_i|, |b_i|$ for $1 \leq i \leq R$.
\end{conj}

\begin{rem}
    If $\SS = \OO[R, C]{1}$ holds in \cref{thm:hlc} (as will be our case), we can absorb the first error term in the second one.
\end{rem}

This conjecture is a quantitative generalization of the well-known $k$-tuples conjecture (due to Hardy and Littlewood, \cite{hardylittlew}), which is precisely the special case $a_i = 1$.

\cref{thm:hlc} is in turn almost a special case ($d = 1$) of the generalized Hardy--Littlewood conjecture presented by Green and Tao in \cite{greentao}. The difference is that we expect a certain uniformity of the error term. Specifically we conjecture that the error term does not depend on the specific $a_i$ and $b_i$ as long as $(\log x)^A \geq |a_i|, |b_i|$. This is done in analogy to the Siegel--Walfisz theorem, which is (except differing error terms) equivalent to the $R = 1$ case of the conjecture.

It is also worthy to note that the conjecture is in perfect accordance with the heuristic \emph{Cramér-model} which asserts that the ``probability'' of a number $n$ being prime is $\prod_{p \leq n} \lfrac1{p} \sim \lfrac1{\log(n)}$. The singular series is then introduced as a correction factor accounting for the fact that the probability of different numbers with certain gaps being prime is not ``independent''.

\begin{lem}\label{lem:sing}
	For all choices of $\alpha,\beta,\gamma$ satisfying the conditions (\ref{eq:ok}) the singular series $\SSA$ corresponding to the linear functions $p(n),q(n),r(n)$ is the same.
\end{lem}
\begin{proof}
	As $R=3$ for all these choices, we just need to show that $\rho(p)$ is independent from $\alpha,\beta,\gamma$ for all primes $p$.
	
	For $p\in\{2,3\}$ we notice that for all choices of $\alpha,\beta,\gamma$ satisfying (\ref{eq:ok}), two of the three polynomials are never divisible by $p$, while the third one gives us exactly one solution of $\prod_{i=1}^{R}(a_in+b_i)=0$~modulo~$p$. Thus $\rho(2)=\rho(3)=1$.
	
	For $p\geq 5$ we find $(p,a_i)=1$ for $i=1,2,3$ and all allowed choices of $\alpha,\beta,\gamma$. By Bezout's identity, each of the polynomials yields us exactly one solution of $\prod_{i=1}^{3}(a_in+b_i)=0$~modulo~$p$. But by \cref{lem:irr}, these have to be pairwise distinct. Thus $\rho(p)=3$.
\end{proof}

This allows us to calculate the singular series
\begin{align*}
	\SSA &=\frac{1-\frac{1}{2}}{\left(1-\frac{1}{2}\right)^3}\cdot\frac{1-\frac{1}{3}}{\left(1-\frac{1}{3}\right)^3} \prod_{p\geq 5\text{ prime}}\left(1-\frac{3}{p}\right)\left(1-\frac{1}{p}\right)^{-3} \\
        &=9\prod_{p\geq 5\text{ prime}}\left(1-\frac{3}{p}\right)\left(1-\frac{1}{p}\right)^{-3}\approx \SSnum  ~.
\end{align*}
The convergence of the product is guaranteed because the sum over $\left( 1 - \lfrac3{p} \right) ( 1 - \lfrac1{p})^{-3} - 1 = (1 - 3p)(p-1)^{-3}$ converges.

We first prove the statement analogous to \cref{thm:bound} 
\begin{thm} \label{thm:bound_}
    Assuming \cref{thm:hlc} we have
    \begin{equation}
    |A_{\leq x}| \sim \frac{\SSA}{2} \frac{x}{(\log x)^3}~,
    \end{equation}
\end{thm}
\begin{proof}
By \cref{lem:exp}, we have
\[
    |\Ax| = \sum_{\beta \geq 1, \gamma \geq 2} \Ax^{1, \beta, \gamma} + \sum_{\alpha \geq 2, \beta \geq 1} \Ax^{\alpha, \beta, 1} ~.
\]
We start with the sum for $\alpha = 1$. In this case all values of $p$ are of the form $p(n) = 3^\beta 2^\gamma n + k$ with $n \in \NN$ and $0 \leq k < 3^\beta2^\gamma$. For $p \in A$ we must also have $q(n) = 2^{\gamma-1} n + \frac{k+1}{2 \cdot 3^\beta}, r(n) = 3^\beta n + \frac{k-1}{2^\gamma}$ prime. 

We set $y = (\log x)^6$ and split the sum into two parts (The lower bound for $\beta$ and $\gamma$ is implicit in all sums):
\begin{equation}
    \sum_{\beta \geq 1, \gamma \geq 2} \Ax^{1, \beta, \gamma} =  \sum_{3^\beta 2^\gamma \leq y} \Ax^{1, \beta, \gamma}
                                                               + \sum_{y < 3^\beta 2^\gamma \leq x} \Ax^{1, \beta, \gamma} \label{eq:part}
\end{equation}
We bound the second sum trivially by
\begin{align*}
    \sum_{y < 3^\beta 2^\gamma \leq x} \Ax^{1, \beta, \gamma} &=
                \sum_{y < 3^\beta 2^\gamma \leq x} \left|\left\{0 \leq n \leq \frac{x-k}{3^\beta 2^\gamma} : p(n), q(n), r(n) \text{ prime} \right\}\right| \\
                                                             &\leq  \sum_{y < 3^\beta 2^\gamma \leq x} \left( \frac{x-k}{3^\beta 2^\gamma} + 1 \right)
                                                             \leq \log_2(x) \log_3(x) \left( \frac{x}{y}  + 1 \right)\\
                                                             &\ll \frac{x}{(\log x)^4}  ~.
\end{align*}
For the first sum in \cref{eq:part} we have
\begin{align*}
    3^\beta 2^\gamma \leq (\log x)^6
                     & \leq \left( \log\!\left(x - (\log x)^6 \right) - 6 \log\log x \right)^7 \\
                     &\leq \left(\log \frac{x-k}{3^\beta 2^\gamma} \right)^7
\end{align*}
for $x$ large enough, allowing us (together with \cref{lem:irr}) to invoke the generalized Hardy--Littlewood-Conjecture (\cref{thm:hlc}) for $C = 7$:
\begin{align}
     \sum_{3^\beta 2^\gamma \leq y} \Ax^{1, \beta, \gamma} &=
         \sum_{3^\beta 2^\gamma \leq y} \left(
                \SSA \frac{x-k}{3^\beta 2^\gamma}\frac1{\left( \log\frac{x-k}{3^\beta 2^\gamma} \right)^3} 
                + o \left( \frac{x-k}{3^\beta 2^\gamma} \frac1{\left( \log\frac{x-k}{3^\beta 2^\gamma} \right)^3}  \right)
         \right) 
         \label{eq:invoke}\\
         &= \SSA \frac{x}{(\log x)^3}\sum_{3^\beta 2^\gamma \leq y} \frac1{3^\beta 2^\gamma} + o \left( \frac{x}{(\log x)^3} \right) \notag
\end{align}
For the summation of the main term we used
\begin{equation*}
\frac{x}{(\log x)^3} \sim \frac{x}{\left( \log(x - (\log x)^6) - 6 \log\log x \right)^3}
                     \geq \frac{x - k}{\left( \log \frac{x-k}{3^\beta 2^\gamma}\right)^3}
                     \geq \frac{x}{(\log x)^3} - (\log x)^6
                     \sim \frac{x}{(\log x)^3}  ~.
\end{equation*}
We can sum all the error terms because the implied constants are absolute (i.\,e.\ independent of $\beta$ and $\gamma$) and furthermore 
\[
    o\!\left(\frac{x-k}{\left( \log\frac{x-k}{3^\beta 2^\gamma} \right)^3} \right) = o\!\left(\frac{x}{( \log x )^3} \right)  ~.
\]
%
Note that $\sum_{\beta \geq 1, \gamma \geq 2} 3^{-\beta} 2^{-\gamma} = \frac14 < \infty$.

It is obvious that the same argument (replacing $\gamma$ by $\alpha$) also works for $\sum_{\alpha \geq 2, \beta \geq 1} \Ax^{\alpha, \beta, 1}$ (since the polynomial for $p$ has leading coefficient $2^\alpha 3^\beta$ instead of $3^\beta 2^\gamma$).
Putting all together, we get
\begin{equation}
    |\Ax| = 2 \SSA \frac{x}{(\log x)^3} \sum_{\substack{\beta \geq 1, \sigma \geq 2\\3^\beta 2^\sigma \leq y}}\frac1{3^\beta 2^\sigma} +
                o\!\left(\frac{x}{(\log x)^3} \right) \label{eq:almost}  ~.
\end{equation}
Taking a closer look at the sum in the last equation, we notice that
\begin{align*}
    \frac14 \geq \sum_{3^\beta 2^\sigma \leq y} \frac1{2^\sigma 3^\beta}
    &\geq \sum_{3^\beta,\,2^\sigma \leq \sqrt{y}} \frac1{2^\sigma 3^\beta} \\
    &= \left( \sum_{\beta \geq 1} \frac1{3^\beta}  + \OO{\frac1{\sqrt{y}}}\right) \left( \sum_{\sigma \geq 2} \frac1{3^\beta} + \OO{\frac1{\sqrt{y}}}\right)\\
    &= \frac14 + \OO{\frac1{(\log x)^{5/2}}} ~,
\end{align*}
which yield the claim when plugged into \cref{eq:almost}.
\end{proof}


Halberstam and Richert \cite{halb} unconditionally prove an upper bound for the quantity of interest in \cref{thm:hlc}:

\begin{thm}[\cite{halb}, Theorem 5.7] \label{thm:halb}
In the Situation of \cref{thm:hlc} (although we no longer need that $(a_i, b_i) = 1$)
set
\begin{equation*}
E = \prod_{i=1}^{R} a_i \cdot \prod_{1\leq r < s \leq R} (a_r b_s - a_s b_r) ~.
\end{equation*}
If $E \neq 0$ we have
\begin{equation}
    \abs{\set{n \leq x : \forall i\colon a_i n + b_i \mathrm{~prime}}}  \leq 2^R R! \SS \frac{x}{(\log x)^R}
        + \OO[R]{x\frac{ \log\log 3x + \log\log 3|E|}{(\log x)^{R+1}}} \label{eq:halb}  ~,
\end{equation}
where $\SS$ is the singular series given in \cref{def:SS}.
\end{thm}
\begin{rem}
Note that if Inequality \ref{eq:halb} did not have the additional constant $2^R R!$ and were an equality this would imply \cref{thm:hlc}. This is due to the fact that the $\log\log 3\abs{E}$ is negligible in the error term given the hypotheses of the conjecture (we could even replace $\log\log 3|E|$ by the logarithm of some polynomial in $a_i, b_i$), so basically the error term of \cref{thm:halb} would be ``good enough'' for \cref{thm:hlc}. Since both the theorem and the conjecture deal with the same quantity, this provides some more evidence in favor of the error term in \cref{thm:hlc}.
\end{rem}

This allows us to prove the unconditional version of \cref{thm:bound_} with a weakened upper bound which is the result analogous to \cref{thm:uncon}.

\begin{thm}\label{thm:uncon_}
    In the situation of \cref{thm:bound_} we have, without assuming any conjectures
    \begin{equation*}
    \limsup_{x \to \infty} \frac{\abs{A_{\leq x}}}{x / (\log x)^3} \leq 24 \SSA ~.
    \end{equation*}
\end{thm}

\begin{proof}
We only have to modify how we proceed at \cref{eq:invoke} in the proof of \cref{thm:bound}. Using the shorthand $\xi = \frac{x - k}{3^\beta 2^\gamma}$ we get from \cref{thm:halb} that
\begin{equation*}
    \sum_{3^\beta 2^\gamma \leq y} \Ax^{1, \beta, \gamma} \leq
        \sum_{3^\beta 2^\gamma \leq y} \left(
        48 \SSA \frac{x-k}{3^\beta 2^\gamma} \frac1{(\log \xi)^3} 
         + \frac1{3^\beta 2^\gamma} \OO[R]{(x-k) \frac{\log\log 3 \xi + \log\log 3|E|}{(\log \xi)^4 }}
        \right) ~.
\end{equation*}
From the definition of $E$ it follows that $\abs{E} \leq 8 (3^\beta 2^\gamma)^9 \leq 8 y^9 \ll x$ so we can neglect it in the error term. And since
we have $\lfrac{(\log\log \xi)}{(\log \xi)^4 } \ll \lfrac{(\log\log x)}{(\log x)^4 } $ for $3^\beta 2^\alpha \leq y$, we get
\begin{equation*}
    \sum_{3^\beta 2^\gamma \leq y} \Ax^{1, \beta, \gamma} \leq
        48 \SSA \frac{x}{(\log x)^3} \sum_{3^\beta 2^\gamma \leq y} \frac1{3^\beta 2^\gamma}  + \OO{\frac{\log\log x}{(\log x)^4}}  ~.
\end{equation*}
Thus we have the same result as before, just with $48\SSA$ instead of $\SSA$ and an (asymptotic) inequality instead of an (asymptotic) equality (and with a better error term). Proceeding as in \cref{thm:bound} therefore gives us the desired result.
\end{proof}


We now move on and consider the density of the set $B$, proving a result analogous to \cref{thm:A0}.

\begin{thm} \label{thm:A0_}
    Assuming \cref{thm:bound_} we have
    \begin{equation*}
    \abs{\allB_{\leq x}} \sim \abs{\allA_{\leq x}} \sim \frac{\SSA}{2} \frac{x}{(\log x)^3} ~.
    \end{equation*}
\end{thm}

\begin{rem}
    Our proof shows that $|B_{\leq x}| \sim |A_{\leq x}|$ as long as $|A_{\leq x}| \sim K x / (\log x)^3$, even if $K$ is not the constant predicted by \cref{thm:hlc}.
\end{rem}

\begin{proof}
We show that $(A \setminus B)_{\leq x} = o\!\left( \lfrac{x}{(\log x)^3} \right)$. For $p, p' \in A$ let $r, q_, r', q'$ denote the corresponding primes from the definition of $A$ in \cref{eq:Adef}. Setting
\begin{equation}
    S = \set{p \in A : \exists p' \in A_{> p},~~ r = r' \vee q = r' \vee p = r' \vee r = q' \vee q = q' \vee p = q'} \label{eq:S} ~,
\end{equation}
we are guaranteed to have $A \setminus B \subseteq S$ (we only have to look at these six cases in \cref{eq:S} because $p' > p > r, q$).

In order to bound the size of $S$ we take a look at
\begin{align*}
    \preS{1} &= \{p \in A: \exists \alpha', \beta' \in \NNp : p' = 2^{\alpha'} 3^{\beta'} r - 1 \text{ prime}\} \\
    \preS{2} &= \{p \in A: \exists \alpha', \beta' \in \NNp : p' = 2^{\alpha'} 3^{\beta'} q - 1 \text{ prime}\} \\
    \preS{3} &= \{p \in A: \exists \alpha', \beta' \in \NNp : p' = 2^{\alpha'} 3^{\beta'} p - 1 \text{ prime}\} \\
    \preS{4} &= \{p \in A: \exists \gamma' \in \NNp : p' = 2^{\gamma'} r + 1 \text{ prime}\}            \\
    \preS{5} &= \{p \in A: \exists \gamma' \in \NNp : p' = 2^{\gamma'} q + 1 \text{ prime}\}            \\
    \preS{6} &= \{p \in A: \exists \gamma' \in \NNp : p' = 2^{\gamma'} p + 1 \text{ prime}\} ~.
\end{align*}
Again, we allow superscripts like $\preS{1}^{\alpha, \beta, \gamma; \alpha', \beta'}$ or $\preS{4}^{\alpha, \beta, \gamma; \gamma'}$ to let $p$ range over $A^{\alpha, \beta, \gamma}$ and to specify choice of $\alpha', \beta'$ or a choice of $\gamma'$ in \cref{eq:S}. For conciseness we may write this as $\preS{i}^{\bm{\mu}}$, having set $\bm{\mu} = (\alpha, \beta, \gamma; \alpha', \beta')$ or respectively $\bm{\mu} = (\alpha, \beta, \gamma; \gamma')$.

Obviously we have the bound $\abs{S_{\leq x}} \leq \sum_{i=1}^{6} \abs{\preS{i}_{\leq x}}$.

For $i = 1, \dots, 6$ we now want to look at $\preS{i}$.
In \cref{lem:gamma} and \cref{lem:alpha}, for fixed $\alpha, \beta, \gamma$ satisfying the conditions (\ref{eq:ok}), we expressed $p, q, r$ as linear integer polynomials in a variable $n$. Therefore $p'$ can also be expressed as a linear integer polynomial $f$ (whose coefficients depend on $\abg$ and, depending on $i$, on $\alpha', \beta'$ or on $\gamma'$) in $n$. One easily sees that the leading coefficient of $f$ only has prime factors $2$ and $3$ and that its constant coefficient is lesser than the leading one.


Now fix $\abg$ such that they satisfy the conditions (\ref{eq:ok}) and also fix $(\alpha', \beta')$ or $\gamma'$ and denote the corresponding tuple by $\bm{\mu}$. We set 
\begin{equation*}
    T = \prod_{\substack{\{g, h\} \subset \{p, q, r, f\} \\ g \neq h}} \res(g, h) ~,
\end{equation*}
where $\res$ denotes the resultant of two polynomials. 
Since no two of the polynomials $p, q, r$ are equal up to sign, $T$ is zero if and only if any of $\res(f, p), \res(f, q)$ or $\res(f, r)$ is zero. This is equivalent to $f$ being an integer multiple of $p, q$ or $r$.
Therefore $T = 0$ implies $\preS{i}^{\bm{\mu}} = \emptyset$, since for any $n \in \NN$ either $p' = f(n)$ is not strictly larger than $p(n)$ or is composite.

Now assume $T \neq 0$. In order to apply \cref{thm:halb} to $\preS{i}$ we set $R = 4$ and look at the linear functions $\{p, q, r, f\}$.
If for a prime $\l \geq 5$ we have $\l \nmid T$ then $p, q, r, f$ have 4 distinct zeros modulo $\l$ (because all leading coefficients are coprime to $\l$), so $\rho_{\{p, q, r, f\}}(l) = 4$. Together with the inequality $1 - R/p \leq (1 - 1/p)^R$ this allows us to give an estimate for the corresponding singular series:
\begin{align*}
    \SS_{\{p, q, r, f\}} &= \prod_{\l \text{ prime}} \frac{1 - \rho_{\{p, q, r, f\}}(l)/p}{(1 - 1/p)^4} \\
                       &\leq \frac1{(1 - 1/2)^4} \frac1{(1 - 1/3)^4} \prod_{\substack{\l \text{ prime} \\ \l \nmid T}} \frac{1 - 4/p}{(1 - 1/p)^4}
                                    \prod_{\substack{\l \text{ prime} \\ \l \mid T}} \frac{1 - 1/p}{(1 - 1/p)^4} \\
                       &\ll \left(\frac{T}{\phi(T)} \right)^3 \ll (\log\log T)^3 ~,
\end{align*}
where $\phi$ denotes the Euler totient function and we use the estimate $\phi(n) \gg \lfrac{n}{\log\log n}$ from \cite[, Theorem 5.6]{tenenbaum}.

In a similar fashion to the proof of \cref{thm:bound} we now proceed to sum over all exponents $\abg$ satisfying (\ref{eq:ok}) and over $\alpha', \beta'$ (for $i=1, 2, 3$) or over $\gamma'$ (for $i=4, 5, 6$). We denote by $M$ the largest (leading) coefficient in the polynomials $p, q, r, f$ and by $k$ the constant coefficient of the according polynomial ($M$ is supposed to depend on the exponents $\alpha, \beta, \dots$ over which we are going to sum).

We have the bounds $M \geq 2^\alpha 3^\beta$ or $M \geq 3^\beta 2^\gamma$, depending on which case we have in (\ref{eq:ok}) and $M \geq 2^{\alpha'} 3^{\beta'}$ or $M \geq 2^{\gamma'}$, depending on $i$.

For a given $x$ we set $y = (\log x)^{8}$ and write $\sum^*$ for a sum over all $\abg$ satisfying (\ref{eq:ok}) and over $\alpha', \beta'$ or $\gamma'$.
We have
\begin{equation*}
\sums{}  \preS{i}^{\bm{\mu}}_{\leq x} = \sums{M \leq y} \preS{i}^{\bm{\mu}}_{\leq x} +
                                                   \sums{y < M \leq x} \preS{i}^{\bm{\mu}}_{\leq x} ~.
\end{equation*}
Both sums range over either $(\alpha, \beta)$ or $(\beta, \gamma)$ as well as over $(\alpha', \beta')$ or over $\gamma'$.
Using our bounds for $M$ we can therefore deduce that the second sum has at most $(\log_2 x)^2 (\log_3 x)^2 \ll (\log x)^4$ summands while the first sum has at most $(\log_2 y)^2 (\log_3 y)^2 \ll (\log \log x)^4$.

Again, we begin with the second sum:
\begin{equation*}
\sums{y < M \leq x} \preS{i}^{\bm{\mu}}_{\leq x} \leq \sums{_{y \leq M \leq x}} \left( \frac{x}{M} + 1 \right)
            \leq (\log x)^4 \left( \frac{x}{y} + 1 \right) \ll \frac{x}{(\log x)^4}
\end{equation*}

In the first sum we can bound the summands with $T = 0$ by any positive number. With this in mind we apply \cref{thm:halb}%
\footnote{$E = 0 \iff T = 0$, hence we only need to apply the theorem in the cases where $E \neq 0$}
to get
\begin{equation*}
    \sums{M \leq y} \preS{i}^{\bm{\mu}}_{\leq x}
            \leq \sums{M \leq y} \left( 2^4 4! \SS_{\{p, q, r, f\}} \frac{x}{M (\log\xi)^4}
                                        + \OO{\xi \frac{\log\log(3\xi) + \log\log(3 \abs{E})}{(\log \xi)^5}}
                                 \right) ~.
\end{equation*}
Since $\abs{E} \leq M^4 \left(2 M^2 \right)^6 \ll (\log x)^{128}$ we can neglect its contribution to the error term. As $\xi \leq x$ and there are only $\OO{(\log \log x)^4}$ summands, the sum of all error terms is $\OO{x/(\log x)^4}$.
Using the bound for the singular series together with $T \leq \left( 2M^2 \right)^{6} \ll  x$
on the main terms therefore gives us
\begin{equation*}
    \sums{M \leq y} \preS{i}^{\bm{\mu}}_{\leq x} \ll \frac{x (\log\log x)^3}{(\log x)^4} \sums{M \leq y} 1 + \OO{\frac{x}{(\log x)^4}}
        \ll \frac{x (\log\log x)^7 }{(\log x)^4}  ~.
\end{equation*}
Here again we have $\log \xi \sim \log x$ because of
\begin{equation*}
x \geq \xi \geq \left( x - (\log x)^8 \right)/(\log x)^8 \sim x ~.
\end{equation*}

Having a bound of $o\!\left( x/(\log x)^3 \right)$ for every $\preS{i}_{\leq x}$ separately, we get such a bound for $S_{\leq x}$, which proves the claim.
\end{proof}

As announced in \cref{sect:main} we finally show that all our main results can be applied to $\allA$ and $\allB$
\begin{lem} \label{lem:same}
    The asymptotic estimates from \cref{thm:bound_}, \cref{thm:uncon_} and \cref{thm:A0_} also hold for $\allA$ and $\allB$, respectively.
    
    This proves Theorems~\ref{thm:bound}, \ref{thm:uncon} and \ref{thm:A0}
\end{lem}

\begin{proof}
First note that the maximality of $B$ implies $\allB \setminus B \subset \allA \setminus A$. We now prove the claim by showing that
$| (\allA \setminus A)_{\leq x} | = o\!\left( \lfrac{x}{(\log x)^3} \right)$. Since all our density estimates were of the order $x/(\log x)^3$ this is enough to prove our claim.

According to the definition of $\allA$ in \cref{eq:AAdef}, to every $p \in \allA \setminus A$ correspond two primes $p, r$ and two exponents $\mu, \nu \in \NNp$. The number of $p \in (\allA \setminus A)_{\leq x}$ with $\mu \geq 2$ are limited by $\sqrt{x} \log_2(x) \log_3(x)$ while the number of $p$ with $\nu \geq 2$ is limited by $\sqrt{x} \log_2(x)$. So $ | (\allA \setminus A)_{\leq x} | = \OO{\sqrt{x} (\log x)^2}$, which yields the claim.

\end{proof}

\appendix
\section{Empirical data}
We calculated $A_{\leq x}$ for $x = 25 \cdot\!10^{12}$. Some results are displayed in \cref{tb:emp}.
In addition to $x/(\log x)^3$ we also compare $|A_{\leq x}|$ to $\operatorname{Li}_3 (x) = \int_{2}^{\infty} \frac{\mathrm{d}x}{(\log x)^3}$ which is asymptotically equivalent to $x/(\log x)^3$, but expected to produce better approximations for ``small'' numbers. 
\begin{table}[h]
\begin{tabular}{rrrc}
    \multicolumn{1}{c}{$x$} & \multicolumn{1}{c}{$| A_{\leq x} |$} & $\frac{| A_{\leq x} |}{x / (\log x)^3}$ &
    $\frac{| A_{\leq x} |}{\operatorname{Li}_3 (x)}$\\ \hline
    $10^{4\phantom{0}}$  & $114$            & $334.70$   & $4.70$ \\
    $10^{6\phantom{0}}$  & $2192$           & $64.40$    & $4.33$ \\
    $10^{8\phantom{0}}$  & $74531$          & $21.90$    & $3.84$ \\
    $10^{10}$            & $3393108$        & $9.96$     & $3.57$ \\
    $10^{12}$            & $183047288$      & $5.37$     & $3.42$ \\
    $25 \cdot\! 10^{12}$ & $3174617502$     & $3.73$     & $3.35$
\end{tabular}
\medskip
\caption{Values of $A_{\leq x}$ compared to the predicted approximations.}
\label{tb:emp}
\end{table}
It seems plausible the limit of the ratios is $\SSA/2 \approx 2.86$. However, due to the slow speed of convergence, this data alone does not allow to make any useful predictions about the limit of any of the the ratios.

\bibliographystyle{amsalpha}
\bibliography{paper}

\end{document}